\documentclass{amsart}
 \usepackage{amsmath}
\usepackage{amssymb}
\usepackage{amsfonts}
 \usepackage{eucal}
      \makeatletter
     \def\section{\@startsection{section}{1}%
     \z@{.7\linespacing\@plus\linespacing}{.5\linespacing}%
     {\bfseries
     \centering
     }}
     \def\@secnumfont{\bfseries}
     \makeatother
\newtheorem{theorem}{Theorem}[section]
\newtheorem{lemma}[theorem]{Lemma}

\newtheorem{proposition}[theorem]{Proposition}

\theoremstyle{definition}

\theoremstyle{remark}
\newtheorem{remark}[theorem]{Remark}

\numberwithin{equation}{section}

\def \a{{\alpha}}
\def \b{{\beta}}
\def \D{{\Delta}}

\def \e{{\varepsilon}}

\def \g{{\gamma}}

\def \l{{\lambda}}

\def \m{{\mu}}

\def \E{{\bf E}\, }

\def \P{{\bf P}}

\def \qq{{\qquad}}
\def \R{{\bf R}}

\def \T{{\bf T}}

\def \Z{{\bf Z}}

\def \dd{{\rm d}}

\def \noi{{\noindent}}

\def\E{{\mathbb E \,}}

\def \T{{\mathbb T}}

\def\P{{\mathbb P}}

\def\R{{\mathbb R}}

\def\Z{{\mathbb Z}}

   \scrollmode
 at 9,8  pt
 
\title[Steinhaus chaose on the polydisc]{$\boldsymbol L^{\boldsymbol 1}$-Norm of Steinhaus chaose on the polydisc}
   
\begin{document}
  \author{Michel  J.\! G. WEBER}
 \address{IRMA, Universit\'e
Louis-Pasteur et C.N.R.S.,   7  rue Ren\'e Descartes, 67084
Strasbourg Cedex, France.
   E-mail:    {\tt  michel.weber@math.unistra.fr}}


\begin{abstract} Let   $J_n\subset[1,n]$, $n=1,2,\ldots$ be  increasing sets   of mutually coprime numbers. Under reasonable conditions on the coefficient sequence $\{c^j_n\}_{n,j}$, we show that 
$$ \lim_{T\to \infty}\frac{1}{T} \int_{0}^T \Big| \sum_{j\in J_n}  c^j_n\,j^{it}\Big| \dd t\sim  \big(\frac{\pi} {2}\sum_{j\in
J_n}  (c^j_n)^2\big)^{1/2} $$
as $n\to \infty$. We also show by means of an elementary device that for all $0<\a<2$, \begin{eqnarray*}
 \lim_{T\to \infty} \Big(\frac{1}{T} \int_{0}^T \big| \sum_{n=1}^N n^{-it}\big|^\a\dd t\Big)^{1/\a} \ge C_\a\, 
\frac{   N^{\frac{1}{2}}} {\big( \log N\big)^{{\frac{1}{\a} -\frac{1}{2} }}}. \end{eqnarray*}
The proof uses Ayyad, Cochrane and Zheng estimate on the number of solutions of the equation
$x_1x_2=x_3x_4$. In the case $\a=1$, this approaches Helson's bound up to a factor $(\log N)^{1/4}$.\end{abstract}

\maketitle
  
 \section{Main Result.}   Let $0<\a<\infty$, let
$ D(it)= \sum_{n=1}^N
x_n n^{-it}$ where  $ x_1, \ldots, x_J$ are complex numbers and  note 
$$\| D\|_\a := \lim_{T\to \infty} \Big(\frac{1}{T} \int_{0}^T \big|D(it)\big|^\a \dd t\Big)^{1/\a}. $$
 Consider first the simple Dirichlet sum
$$\D(it)= \sum_{j=1}^J x_j p_j^{-it}, $$
where $p_j$ denotes the $j$-th prime.  On the polycircle, namely the infinite circle  
$\T^\infty$ ($\T= \R\backslash\Z$)
equipped with the infinite Haar measure,  the monomials $p_j^{-it}$ are, by the Bohr correspondance principle, independent identically distributed  Steinhaus random
variables. As the flow $\mathcal T_t z= (p_1^{ it}z_1,p_2^{ it}z_2, \ldots) $    is uniquely ergodic, it follows from Birkhoff's theorem that one  can use a version of Khintchine's inequality for Steinhaus random variables to estimate their $L^1$-norm. By Theorem 1 in K\"onig \cite{Ko},
for
$0<\a<\infty$,  
\begin{equation} \label{kiko} c_\a\, \|x_j\|_2\le  \| \D\|_\a\le C_\a \, \|x_j\|_2 .
\end{equation}
And the constants $c_\a, C_\a$ are explicited.
 In particular, 
\begin{equation} \label{k2} c_1\,\sqrt J\le  \|\sum_{j=1}^J  p_j^{-it}\|_1\le C_1 \, \sqrt J .
\end{equation}

Consider now a subset $J$ of $\langle p_1, \ldots, p_r\rangle$ composed with mutually coprimes numbers. The more general Dirichlet sum    
$$\D(it)= \sum_{j\in J}  x_{ j}  j^{-it}$$
is   a sum of independent  non identically distributed  Steinhaus chaose. To our knewledge, there is no corresponding version of   Khintchin's inequality. 
A good substitute turns up the application of a suitable version of
the central limit theorem for independent  complex-valued   random variables. 
\vskip 3 pt We obtain the following result which is even more precise in the simple case  considered above. 
\begin{theorem}  \label{t1} For   $n=1,2,\ldots$, let   $J_n\subset[1,n]$ be a set   of mutually coprime numbers,     and
assume that $\#(J_n)\uparrow \infty$ with $n$. Let $\{c^j_n, j\in J_n\}$, $n\ge 1$ be real numbers.
 Assume that the two following conditions are satisfied,
\begin{eqnarray*} (1) & &\sum_{j\in
J_n}  (c^j_n)^2\uparrow\infty ,
\cr (2) & & \max_{   j\in J_n }|c^j_n|=  o\Big(  \big(\sum_{j\in
J_n}  (c^j_n)^2\big)^{1/2} \Big).\end{eqnarray*}   Then 
$$ \big\|\sum_{j\in J_n}  c^j_n\,j^{it}\big\|_1\sim  \Big(\frac{\pi} {2}\sum_{j\in
J_n}  (c^j_n)^2\Big)^{1/2} $$
when $n$ tends to infinity.
 \end{theorem}  
 \begin{remark}  Condition (1) imposes restrictions. In the simplest case when
$J_n= \mathcal P\cap [1,n]$, where $\mathcal P $ is the set of all primes and   $c^j_n= j^{-\a}$, condition (1) holds only if
$\a<1/2$, and condition (2) trivially holds.
\end{remark}
  
\section{Proof.}
The proof  of Theorem \ref{t1} uses Bohr's correspondance.   Recall some classical facts. Let  $z_j= e^{2i\pi
t_j}$ denotes  the 
$j$-th coordinate
  on the polycircle $\T^\infty$. Bohr's correspondance  is the mapping  
$$k=p_1^{v_{p_1}(k)}\ldots p_r^{v_{p_r}(k)}\ \longrightarrow \  z_1^{v_{p_1}(k)}  \ldots \, z_{r}^{v_{p_1}(k)}$$ 
  where $  v_{p }(k)$ is the
$p$-valuation of $k$. 

 Let
$\tilde Dz=\sum_{k=1}^n z_1^{v_{p_1}(k)}\ldots z_r^{v_{p_r}(k)}$ be  the Bohr's image of $D(t)$.  On the polycircle, the action of the flow $\mathcal T_t z= (p_1^{ it}z_1,p_2^{ it}z_2, \ldots) $     on
monomials  reads as
$$\mathcal T_t\big( z_1^{n_1}\ldots z_{r}^{ n_r} \big) = \ p_1^{it n_1}\ldots p_r^{itn_r}z_1^{n_1}\ldots z_{r}^{ n_r}  . $$
As $\mathcal T_t$ is uniquely ergodic, it follows from Birkhoff's theorem that for polynomials
$$\lim_{T\to \infty} \frac{1}{T}\int_{-T}^T \Big|\sum_{\nu=(n_1, \ldots, n_r)}a_{\nu}p_1^{it n_1}\ldots
p_r^{itn_r}\varsigma_1^{n_1}\ldots \varsigma_{r}^{ n_r}\Big|^\a\dd t  =\int_{\T^\infty} |\tilde D\varsigma |^\a\dd \varsigma
 $$
for {\it all} $z=(z_1,z_2,\ldots)\in \T^\infty$. Taking $z_1=\ldots= z_r=1$ in the left hand side
gives
$$\lim_{T\to \infty} \frac{1}{T}\int_{-T}^T \Big|\sum_{\nu=(n_1, \ldots, n_r)}a_{\nu}p_1^{it n_1}\ldots
p_r^{itn_r} \Big|^\a\dd t  = \int_{\T^\infty} |\tilde D\varsigma |^\a\dd \varsigma.$$
Hence 
\begin{equation} \label{babo} \|D\|_\a := \lim_{T\to \infty} \Big(\frac{1}{T}\int_{-T}^T |D(t)  |^\a\dd t \Big)^{1/\a} = \Big(\int_{\T^\infty} |\tilde
D\varsigma |^\a\dd \varsigma\Big)^{1/\a}.
\end{equation}
\vskip 2 pt
Our approach for proving Theorem \ref{t1} is indirect. We will use the central limit theorem,
and more precisely   the lemma below.
\vskip 2pt  Let $\xi^j = (\xi^j_1,  \xi^j_2)$, $j=1,2,\ldots$ be independent  $\R^2$-valued random vectors with mean vectors $\m^j= 0$ and
covariance matrices
$\Gamma^j  = (\E
\xi^j_u\xi^j_v)_{u,v}$. Let   $S_n= \sum_{j=1}^n \xi^j$.    
 Let also $T_n= \sum_{j=1}^n( \xi^j_1+ i\xi^j_2)$. 
Put 
\begin{eqnarray*}\theta_n \ :=\  \E|T_n|^2
\ =\ \sum_{j=1}^n \E 
(\xi^j_1) ^2+\sum_{j=1}^n \E( \xi^j_2)^2.   \end{eqnarray*}
The lemma above shows that under the CLT, the ratio $ { \|T_n\|_1}/{\|T_n\|_2}$  tends to a positive limit. 
 \begin{lemma} Assume that
 $$ \frac{ S_n}{\sqrt {\theta_n}}   \ \buildrel{\mathcal D}\over \longrightarrow \ \mathcal N(0, \Gamma),$$
where $\Gamma$ is regular. Then 
\begin{eqnarray*} \lim_{n\to \infty}\E \frac{ |T_n|}{\sqrt {\theta_n}}&=&   \E (g^2_1+ g^2_2)^{1/2}, \end{eqnarray*}
where $(g_1, g_2)\buildrel{\mathcal D}\over = \mathcal N(0, \Gamma)$. \end{lemma}
 \begin{proof}
   By    Cauchy-Schwarz's inequality, next Tchebycheff's inequality
\begin{eqnarray*}\Big(\E \frac{ |T_n|}{\sqrt {\theta_n}}\cdot \chi\{\frac{ |T_n|}{\sqrt {\theta_n}}> M\}\Big)^2&\le&\Big(  \E \frac{
|T_n|^2}{  {\theta_n}}\Big)
\P\{\frac{ |T_n|}{\sqrt {\theta_n}}> M\}=\P\{\frac{ |T_n|}{\sqrt {\theta_n}}> M\}\cr  & \le &\frac{1}{M^2}  \E \frac{ |T_n|^2}{ 
{\theta_n}} = \frac{1}{M^2}  
 .\end{eqnarray*}
 Hence
 $$\Big|\E \frac{ |T_n|}{\sqrt {\theta_n}}-\E \frac{ |T_n|}{\sqrt {\theta_n}} \cdot \chi\{\frac{ |T_n|}{\sqrt {\theta_n}}\le  M\}\Big|\le
\frac{ 1   }{M } .$$
    
Let $ g= g_1+ig_2$.  Let $\e >0$ and choose $M$ so that  $\frac{1  }{M }+  \E |g| \chi\{|g|>   M\}\le \e$. 
Then
\begin{eqnarray*}\Big|\E \frac{ |T_n|}{\sqrt {\theta_n}}-\E |g|\Big|&\le &\frac{ 1   }{M }+  \E |g| \chi\{|g|>   M\}+\Big|\E
\frac{ |T_n|}{\sqrt {\theta_n}}
\chi\{\frac{ |T_n|}{\sqrt {\theta_n}}\le  M\}-\E |g| \chi\{|g|\le  M\}\Big| 
\cr &\le &\e+\Big|\E \frac{ |T_n|}{\sqrt {\theta_n}} \chi\{\frac{ |T_n|}{\sqrt {\theta_n}}\le  M\}-\E |g| \chi\{|g|\le  M\}\Big|
.\end{eqnarray*}
 It follows from the assumption made that for all reals $t$,
 $$\lim_{n\to \infty} \P \{\frac{ |T_n|}{\sqrt
{\theta_n}}>t\}=\P \{|g|>t\}. $$
Now, $M$ being fixed, by the dominated convergence theorem
\begin{eqnarray*}  \E \frac{ |T_n|}{\sqrt {\theta_n}} \chi\{\frac{ |T_n|}{\sqrt {\theta_n}}\le  M\} = \int_{0}^M\P \{\frac{ |T_n|}{\sqrt
{\theta_n}}>t\}\dd t\to
\int_{0}^M\P \{|g|>t\}\dd t=\E |g| \chi\{|g|\le  M\},\end{eqnarray*} as $n$ tends to infinity. So that
\begin{eqnarray*}\Big|\E \frac{ |T_n|}{\sqrt {\theta_n}}-\E |g|\Big|&\le & 2\e,\end{eqnarray*}
for $n$ large enough. Therefore
\begin{eqnarray*} \lim_{n\to \infty}\E \frac{ |T_n|}{\sqrt {\theta_n}}&=&\E |g| =  \E \sqrt{g^2_1+ g^2_2}. \end{eqnarray*}
\end{proof}
\begin{proof}[Proof of Theorem \ref{t1}]  
Let $r=r_n =\pi(n)$. Let $j\in J_n$, $j=p_1^{v_{p_1}(j)}\ldots p_r^{v_{p_r}(j)}$. 
Let 
$$ \xi_n^j= z_1^{v_{p_1}(j)}  \ldots \, z_{r}^{v_{p_1}(j)}=  \exp\big\{2i\pi  \sum_{h=1 }^r v_{p_h}(j)  t_h  \big\}.$$ 
Set also 
  $ X_n^{j,1}= \Re(\xi_n^j)$, $X_n^{j,2}=  \Im(\xi_n^j) $  and 
\begin{eqnarray*} X_n^j =c^j_n(X_n^{j,1},X_n^{j,2}),\qq S_n = \sum_{j\in J_n}  X_n^j,\qq 
 \Gamma=c_j \, \Big(\begin{matrix} 0 & 1/2\cr 1/2 & 0\end{matrix}\Big).
\end{eqnarray*} 
Notice that
 $ \langle X_n^{j,1},X_n^{j,2}\rangle=0$, 
   $\| X_n^{j,1} \|_2^2=\| X_n^{j,2} \|_2^2=1/2$, and  consequently 
$${\bf Cov} (X_n^j)= c_j^2\, \Gamma .$$ 
 Indeed, 
\begin{eqnarray*}  & & 2\int_{\T^r}  \cos \big(2 \pi  \sum_{h=1 }^r v_{p_h}(j)  t_h \big)   \sin \big(2 \pi  \sum_{h=1 }^r
v_{p_h}(j)  t_h \big)   
\dd t_1\ldots  \dd t_r
\cr &= &   \int_{\T^r}      \sin \big(4 \pi   \sum_{h=1 }^r v_{p_h}(j)  t_h \big)   
 \dd t_1\ldots  \dd t_r
=  \Im\Big\{\int_{\T^r}     \exp \Big(4 i\pi  \sum_{h=1 }^r v_{p_h}(j)  t_h \Big)    \dd t_1\ldots  \dd 
t_r \Big\}\cr &=& \Im \Big\{\prod_{h=1}^r\int_0^1    \exp \Big(4i \pi  v_{p_h}(j)t_h  \Big)     \dd t_h\Big\}=  0.
\end{eqnarray*}  
And
\begin{eqnarray*}   2\int_{\T^r}  \cos^2 \big(2 \pi   \sum_{h=1 }^r v_{p_h}(j)  t_h \big)       \dd t_1\ldots  \dd  t_r
  &= & 1+
\int_{\T^r}   \cos  
\big(4 \pi  \sum_{h=1 }^r v_{p_h}(j)  t_h \big)       \dd t_1\ldots  \dd  t_r
 \cr &= & 1+ \frac{1}{2} \Re \Big\{\prod_{h=1}^r\int_0^1   
\exp  \big(4 \pi  \sum_{h=1 }^r v_{p_h}(j)  t_h  \big)     \dd t_h\Big\}
 \cr &= &1.
\end{eqnarray*} 
 
    Consider the arrays
\begin{eqnarray}\label{ar} Y_n^{j,i} = \frac{X_n^{j,i}}{\sqrt {B_n}} , \qq j\in J_n ,\ i=1,2,\ n\ge 1.\end{eqnarray}  
Since $ \E (|Y_n^{j,i}|\wedge 1)\le \frac{c_n^{j,i}}{\sqrt {B_n}}$, we deduce from assumption (2) that
$$\sup_j \E (|Y_n^{j,i}|\wedge 1)\to 0, $$
as $n$ tends to infinity, so that $(Y_n^{j,i})$, $i=1,2$ are null arrays (\cite{Ka}, p.65). 
\vskip 2 pt 
 It follows from  Theorem 4.15 in \cite{Ka} (see also Exercise 22) that, 
$$ \frac{ S_n}{\sqrt {B_n}}    \ \buildrel{\mathcal D}\over \longrightarrow
\ \mathcal N  (0, \Gamma),$$ if and only if
 \begin{eqnarray*} 
{\rm (i)} && \hbox{$\sum_{j\in J_n} \P\{\frac{ |c^j_n|}{\sqrt {B_n}}|X_n^{j,i}| >\e  \}\to 0$ for all $\e>0$, $i=1,2$;}\cr
{\rm (ii)} && \hbox{$\frac{1}{\sqrt{ B_n}}\sum_{j\in J_n} c^j_n\,\E { X_n^{j,i} } \cdot \chi \{\frac{ |c^j_n|}{\sqrt {B_n}}|X_n^{j,i}| \le
1\}\to b_i$},  \ i=1,2 ;
\cr {\rm (iii)} && \hbox{$\frac{1}{   { B_n}}\sum_{j\in J_n}{\rm Var}\big( { c^j_n   X_n^{j,i} } \cdot \chi \{\frac{ |c^j_n|}{\sqrt
{B_n}} |X_n^{j,i}|
\le  
1\}\big)\to d_i$},\  i=1,2 .
\end{eqnarray*}
Then, 
\vskip 2pt \noi (a) For (i), the sum vanishes as soon as $B_n\ge \max_{j\in J_n}(c^j_n)^2/ \e^2$, which is true for $n\ge n_\e$   by assumption (2). 
\vskip 2 pt
\noi (b) By centering, for $n\ge 1$,
$$\E { X_n^{j,i} } \cdot \chi \{|c^j_n|\,|X_n^{j,i}| \le   \sqrt{ B_n}\}=\E { X_n^{j,i} }   =0,$$
so that (ii) holds with $b_i\equiv0$, $i=1,2$. 
\vskip 2 pt
\noi (c) Now 
\begin{eqnarray*}& &{\rm Var}\big( { c^j_n X_n^{j,i} } \cdot \chi \{|c^j_n||X_n^{j,i}| \le   \sqrt{ B_n}\}\big)\cr &=& (c^j_n)^2\E   
 (X_n^{j,i} )^2 
\cdot \chi
\{|c^j_n||X_n^{j,i}|
\le  
\sqrt{ B_n}\}-( c^j_n \E    X_n^{j,i} \cdot  
\chi
\{|c^j_n||X_n^{j,i}| \le   \sqrt{ B_n}\})^2 
\cr &=& (c^j_n)^2\, \E   (X_n^{j,i})^2   =(c^j_n)^2/2, 
\end{eqnarray*}
so that,
(iii) holds with $d_i\equiv1/2$. Hence,  
$$ \frac{ S_n}{\sqrt{ B_n}}   \ \buildrel{\mathcal D}\over \longrightarrow \ \mathcal N (0,\Gamma).$$
 Using the lemma above, we deduce that
 \begin{eqnarray*} \lim_{n\to \infty}  \frac{ 1}{\sqrt{ B_n}}\E |\sum_{j\in J_n} \xi_n^j|&=&   \E(g^2_1+ g^2_2)^{1/2}. \end{eqnarray*}
\end{proof}

\section{Additionnal Results.}
Recall that $ D(it)= \sum_{n=1}^N
n^{-it}$ and that we have noted
$$\| D\|_r := \lim_{T\to \infty} \Big(\frac{1}{T} \int_{0}^T \big| \sum_{n=1}^N n^{-it}\big|^r\dd t\Big)^{1/r}. $$
Elementary approachs   allow to  already  estimate quite well the norms $\| D\|_r $, $0<r<2$. This is the purpose of this section. We actually don't know whether these can be refined and permit to provide more precise results, but their utilization here is certainly not optimal. The study of the case $r=1$ is known to be  very difficult, and a  central result is   naturally Helson's estimate  \cite{H} (Theorem p.82 \& Appendix), $\| D\|_1\ge \big( \sum_{n=1}^N \frac{1}{d(n)}\big)^{1/2}$.  
As (\cite{W}, (3.10)), $ \sum_{n\le x}  \frac{1}{d(n)}\sim \frac{x}{\sqrt{ \log x}}$, this yields
$$ \| D\|_1\ge C  \frac{ N^{1/2}} {  (\log N)^{1/4}}.    $$
  \begin{proposition}\label{p1}  For all $0<r<2$, 
\begin{eqnarray*}
 \|D\|_r \ge C_r
\frac{   N^{\frac{1}{2}}} {\big( \log N\big)^{{\frac{1}{r} -\frac{1}{2} }}}. \end{eqnarray*} 
 \end{proposition}
 
For the proof of Proposition \ref{p1}, we need the following lemma.
\begin{lemma}\label{l1} Let $F\ge 0$ defined on an interval $I$ and  $0<r<2$. Let $\l$ denotes the normalized Legesgue measure on $I$. Then,
\begin{eqnarray*}\int_I F^r \dd \l&\ge &\frac{\big(\int_I F^{2} \dd \l\big)^{\frac{4-r}{2}}}{\big(\int_I F^{4} \dd \l\big)^{1-\frac{r}{2} }}.
\end{eqnarray*}
\end{lemma}
 \begin{proof} Let  $\b=\frac{2r}{4-r} $, $\a=\frac{8-4r}{4-r}$ and notice that $0<\b<r$,   $\a+\b=2$. By applying H\"older's inequality with $p=\frac{4-r}{2-r}$, $q=\frac{{4-r}}{2 }$, we get
\begin{eqnarray*}\int_I F^2\dd \l&=&\int_I F^\a F^\b\dd \l
\ \le \ \big(\int_I F^{\a p} \dd \l\big)^{1/p} \big(\int_I F^{\b q} \dd \l\big)^{1/q} 
\cr &=& \big(\int_I F^{4} \dd \l\big)^{\frac{2-r}{4-r} } \big(\int_I F^{r} \dd \l\big)^{\frac{2 }{4-r} } \end{eqnarray*}
Hence the claimed inequality.
\end{proof}
We also need the following estimate (\cite{ACZ}, Theorem 3)
\begin{lemma}\label{l2} The number $N(B)$ of integers solutions of the equation $x_1x_2 =x_3x_4$ with $1\le x_i\le B$, $1\le i\le 4$, is given by 
$$ N(B) = \frac{12}{\pi^2}B^2\log B + CB^2 + \mathcal O \big( B^{19/13}\log^{7/13} B\big), $$
where $C= 2/\pi^2(12\g -(36/\pi^2)\zeta'(2)-3)-2$, $\g$ is Euler's constant and  $\zeta'(2)= \sum_{n=1}^\infty  (\log n) /n^2 $. \end{lemma}\begin{proof}[Proof of Proposition \ref{p1}]
 Lemma \ref{l1} with $F(t)=\big| \sum_{n=1}^N
n^{-it}\big|$ gives  for $0<r<2$, 
\begin{eqnarray}\label{be}
   \frac{1}{T} \int_{0}^T \big| \sum_{n=1}^N n^{-it}\big|^r \dd t &\ge&
      \frac{ \Big(\frac{1}{T}\int_{0}^T \big| \sum_{n=1}^N n^{-it}\big|^2 \dd t \Big)^{\frac{4-r}{2}}} {\Big(\frac{1}{T}\int_{0}^T \big| \sum_{n=1}^N n^{-it}\big|^4  \dd t\Big)^{{1-\frac{r}{2} }}}. \end{eqnarray} 
By Theorem 1 p.128 of \cite{M}, for all $N>0$, $T>0$,
\begin{eqnarray*}
 \frac{1}{T}\int_{0}^T \big| \sum_{n=1}^N n^{-it}\big|^2 \dd t &= & N + \theta \frac{N^2}{T}
    , \end{eqnarray*} 
for some real $\theta$ with $|\theta|\le1$.
And by Lemma \ref{l2}, \begin{eqnarray*}
 \int_{0}^T \Big| \sum_{n=1}^N n^{-it}\Big|^4 \dd t &=  &  \int_{0}^T  \sum_{1\le m,n,\m,\nu\le N}   \Big(\frac{\m\nu}{mn}\Big)^{it} \dd t\cr 
 &=& T\sum_{1\le m,n,\m,\nu\le N\atop \m\nu =mn }1 +     \sum_{1\le m,n,\m,\nu\le N\atop \m\nu \not=mn}   \mathcal O_{m,n}(1)
\cr   &=& \frac{12}{\pi^2}T \big(1+  \mathcal O (1)) N^2\log N  +     \mathcal O (1).
      \end{eqnarray*} 

By reporting, we get   for $T\ge N^2$, 
\begin{eqnarray*}  \frac{1}{T} \int_{0}^T \big| \sum_{n=1}^N n^{-it}\big|^r \dd t &\ge& C_r
      \frac{ N^{\frac{4-r}{2}}} {\big(N^2 \log N\big)^{{1-\frac{r}{2} }}}\ = \ C_r
      \frac{ N^{2-\frac{r}{2}}N^{-2+r} } {\big( \log N\big)^{{1-\frac{r}{2} }}}\ = \      \frac{ C_r
 N^{\frac{r}{2}}} {\big( \log N\big)^{{1-\frac{r}{2} }}}. \end{eqnarray*} 
      Hence, it follows that
      \begin{eqnarray*}
 \|F\|_r \ge C_r
\frac{   N^{\frac{1}{2}}} {\big( \log N\big)^{{\frac{1}{r} -\frac{1}{2} }}}, \end{eqnarray*} 
for all $0<r<2$. \end{proof}
\begin{remark} Consider the general Dirichlet sums $ \sum_{n=1}^N a_nn^{-it}$. Then, similarly,
\begin{eqnarray*}
 \Big\| \sum_{n=1}^N a_nn^{-it}\Big\|_1  &\ge  & C\frac{\big(\sum_{n\le N }|a_n|^2 \big)^{3/2}}{\big(N^2\sum_{n,\nu\le N }\frac{|a_n|^2|a_\nu|^2}{[n,\nu]}\big)^{1/2}} .
 \end{eqnarray*}
 Indeed, writing $\big(\sum_{n=1}^N a_n n^{-it} \big)^2= \sum_{n,\nu=1}^N a_n\overline{a}_\nu (\nu n)^{-it}= \sum_{m=1}^{N^2} b_m m^{-it}$, we have
\begin{eqnarray*}
|b_m|&\le & \sum_{n|m\atop n\le N, \frac{m}{n} \le N} |a_n||{a}_{\frac{m}{n}}|\cr
&\le & \big( \sum_{n|m\atop n\le N}|a_n|^2\big)^{1/2}\big( \sum_{n|m\atop \frac{m}{n} \le N} |a_{\frac{m}{n}}|^2\big)^{1/2}
\ =\ \sum_{\nu|m\atop \nu\le N}|a_\nu|^2. \end{eqnarray*}
Hence, 
\begin{eqnarray*}
\sum_{m=1} ^{N^2} |b_m|^2&\le &\sum_{m=1} ^{N^2}\big( \sum_{\nu|m\atop \nu\le N}|a_\nu|^2\big)^2\ =\ \sum_{m=1} ^{N^2}\big( \sum_{n,\nu\le N\atop n|m , \nu |m}|a_n|^2|a_\nu|^2\big)
\cr &=&\sum_{n,\nu\le N }|a_n|^2|a_\nu|^2 \Big( \sum_{1\le m\le N^2\atop n|m , \nu |m}1 \Big)\ \le \ N^2\sum_{n,\nu\le N }\frac{|a_n|^2|a_\nu|^2}{[n,\nu]} \end{eqnarray*}
It follows that for $T$ large
\begin{eqnarray*}
 \frac{1}{T}\int_{0}^T \big| \sum_{n=1}^N a_nn^{-it}\big|^4 \dd t =(1+o(1) )\sum_{m=1} ^{N^2} |b_m|^2   \le (1+o(1) )N^2\sum_{n,\nu\le N }\frac{|a_n|^2|a_\nu|^2}{[n,\nu]}. \end{eqnarray*} 
Therefore, 
\begin{eqnarray*}
 \frac{1}{T}\int_{0}^T \Big| \sum_{n=1}^N a_nn^{-it}\Big| \dd t &\ge  & (1+o(1) )\frac{\big(\sum_{n\le N }|a_n|^2 \big)^{3/2}}{\big(N^2\sum_{n,\nu\le N }\frac{|a_n|^2|a_\nu|^2}{[n,\nu]}\big)^{1/2}} . \end{eqnarray*} 
 And so 
\begin{eqnarray*}
 \Big\| \sum_{n=1}^N a_nn^{-it}\Big\|_1  &\ge  & (1+o(1) )\frac{\big(\sum_{n\le N }|a_n|^2 \big)^{3/2}}{\big(N^2\sum_{n,\nu\le N }\frac{|a_n|^2|a_\nu|^2}{[n,\nu]}\big)^{1/2}} .
 \end{eqnarray*}
\end{remark}

\vskip 8 pt
 \noi{\it Final Remark.} Here we give another application of the argument used to prove Theorem \ref{t1},  namely we show that 
$$ \sup_n\sup_{k_1<
\ldots<k_n}\frac{1}{\sqrt n}\,\big\|\sum_{k=1}^n e^{ 2\pi k_jx}\big\|_1 \ge   ( {\pi}/{2} )^{1/2}.$$
 Let $ 
X_1,X_2,\ldots$ be iid r.v.'s,  
$S_n = X_1 +\ldots+ X_n$. Assume that $\E X_1= 0$, $\E X_1^2=1$. Let $g\sim \mathcal N(0,1)$.  It is plain that
\begin{equation} \lim_{n\to \infty}  \frac{\|S_n\|_1}{\sqrt n}=  \E |g|=\big( {\pi}/{2}\big)^{1/2}.\end{equation} 
 (Approximate $ \E  \frac{| S_n |}{\sqrt n}$ by  $\E\frac{| S_n |}{\sqrt n} \chi\{  \frac{| S_n |}{\sqrt n} \le  M\}$ and apply the CLT to     $  \E\frac{| S_n |}{\sqrt n} \chi\{  \frac{| S_n |}{\sqrt n} \le  M\} $) 
Now let $a\ge 2$, $n_k = a^{m_k}$ where $m_1<m_2<\ldots$ are integers.  By Remark 1 in  \cite{B}, 
one can enlarge $[0,1[$ so that on the new probability space there exists an i.i.d.  sequence $\{\theta_j, j\ge 1\}$ such that 
 $  |\cos 2\pi n_kx -\theta_k | \le 8\pi {n_k}/{n_{k+1}}$ a.\,s. 
Let $T_n(x) =\sum_{k=1}^n\cos 2\pi n_kx$, $S_n=\sum_{k=1}^n \theta_k$.
It follows that 
$$ |T_n  -S_n   | \le C \sum_{k=k_1}^n {n_k}/{n_{k+1}}\qq \hbox{a.s.}$$ 
Hence
 $ \| T_n  -S_n   \|_1 \le C \sum_{k=k_1}^n {n_k}/{n_{k+1}} $.  
Therefore,
 $ \| T_n  \|_1\ge    \|S_n   \|_1- C \sum_{k=k_1}^n {n_k}/{n_{k+1}}$. This implies that  
$$\sup_n\sup_{k_1< \ldots<k_n}\big\|\sum_{k=1}^n e^{ 2\pi k_jx}\big\|_1/\sqrt n\ge \liminf_{n\to \infty}  { \| T_n  \|_1}/{\sqrt
n}\ge   ( {\pi}/{2} )^{1/2}.$$ 
(This provides a short proof of a result of Aistleitner  in arXiv:1211.4640.)

\end{document}